\documentclass{ip-journal}
\usepackage{lineno,hyperref}
\usepackage[utf8]{inputenc} \newtheorem{theorem}{Theorem}[section]
\newtheorem{lemma}[theorem]{Lemma}

\numberwithin{equation}{section}
\allowdisplaybreaks

\begin{document}
	
	%
	%
	%
	%
	%
	%
	%
	%
	%

	\title[Transmission problem with delay and weight]
	{A transmission problem for waves under time-varying delay and nonlinear weight}

	\author[Carlos A. S. Nonato]{Carlos A. S. Nonato}
	
	\address{%
		Federal University of Bahia, Mathematics Departament,
		Salvador, 40170-110,  Bahia, Brazil}
	\email{carlos.mat.nonato@hotmail.com}
	
	\author{Carlos A. Raposo }
	\address{Federal University of S\~ao Jo\~ao del-Rei, Mathematics Departament,
		S\~ao Jo\~ao del-Rei, 36307-352, Minas Gerais, Brazil}
	\email{raposo@ufsj.edu.br}
	
	\author{Waldemar D. Bastos}
	\address{S\~ao Paulo State University, Mathematics Departament,
		S\~ao Jos\'e do Rio Preto, 15054-352, S\~ao Paulo, Brazil}
	\email{waldemar.bastos@unesp.br}
	\subjclass{Primary 35L20; 35B40; Secondary 93D15}
	
	\keywords{Transmission problem, Time-variable delay, Nonlinear weights, Exponential stability}
	
	
	\begin{abstract}
This manuscript focus on in the transmission problem for one dimensional waves with nonlinear weights on the frictional damping and time-varying delay. We prove global existence of solutions using Kato's variable norm technique and we show the  exponential stability by the energy method with the construction of a suitable Lyapunov functional.
	\end{abstract}
	
	\maketitle

\section{Introduction}

In this paper we investigate global existence and decay properties  of solutions for a transmission problem for waves with nonlinear weights and time-varying delay. We consider the following system
\begin{equation}\label{EQ1.1}
\begin{gathered}
u_{tt}(x,t) - au_{xx}(x,t) + \mu_1(t)u_t(x,t) + \mu_2(t)u_t(x,t-\tau(t)) = 0 \,\, \mbox{in } \Omega \times ]0, \infty[,\\
v_{tt}(x,t) - bv_{xx}(x,t) = 0 \,\, \mbox{in } ]L_1, L_2[ \times ]0, \infty[,
\end{gathered}
\end{equation}
where $0 < L_1 < L_2 < L_3$, $\Omega = ]0, L_1[ \cup ]L_2, L_3[$ and $a$, $b$ are positive constants.

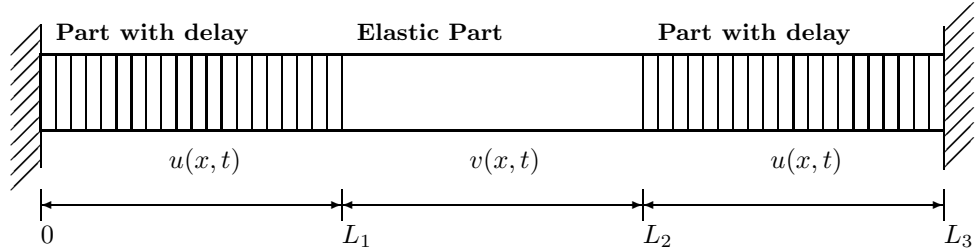
\setlength{\unitlength}{1cm}
\begin{center}
\begin{picture}(12.5,4.5)
\put (0,2){\framebox(12,1){}}
\put (0,1.5){\line(0,1){1.9}}
\put (12,1.48){\line(0,1){1.9}}
\multiput (-0.4,1.2)(0,0.2){10}{\line(1,1){0.4}}
\multiput (12,1.5)(0,0.2){10}{\line(1,1){0.4}}
\multiput (0,2)(0.2,0){21}{\line(0,1){1}}

\put(8,2){\line(0,1){1}}
\multiput (8,2)(0.2,0){21}{\line(0,1){1}}

\put (0.2,3.2){{\small\bf Part with delay }}
\put (4.2,3.2){{\small\bf Elastic Part }}
\put (8.2,3.2){{\small\bf Part with delay }}
\put (2,1){\vector(-1,0){2}}
\put (2,1){\vector(1,0){2}}
\put (6,1){\vector(-1,0){2}}
\put (6,1){\vector(1,0){2}}
\put (10,1){\vector(-1,0){2}}
\put (10,1){\vector(1,0){2}}

\multiput (0,0.8)(4,0){4}{\line(0,1){0.4}}
\put (0,0.5){$0$}
\put (4,0.5){$L_1$}
\put (8,0.5){$L_2$}
\put (12,0.5){$L_3$}
\put(1.7,1.5){$u(x,t)$}
\put(5.7,1.5){$v(x,t)$}
\put(9.7,1.5){$u(x,t)$}
\end{picture}
 \end{center}

The system \eqref{EQ1.1} is subjected to the transmission conditions
\begin{equation}
\begin{gathered}
u(L_i, t) = v(L_i, t), \quad i= 1,2 \\
au_x(L_i, t) = bv_x(L_i, t), \quad i= 1,2,
\end{gathered} \label{EQ1.2}
\end{equation}
the boundary conditions
\begin{equation}\label{EQ1.3}
u(0, t) = u(L_3, t) = 0
\end{equation}
and initial conditions
\begin{equation}
\begin{gathered}
u(x,0) = u_0(x), \quad u_t(x,0) = u_1(x) \quad \mbox{on } \Omega, \\
u_t(x, t-\tau(0)) = f_0(x, t-\tau(0)) \quad \mbox{in } \Omega \times ]0, \tau(0)[, \\
v(x,0) = v_0(x), \quad v_t(x,0) = v_1(x) \quad \mbox{on } ]L_1, L_2[,
\end{gathered} \label{EQ1.4}
\end{equation}
where the initial datum $\left(u_0, u_1, v_0, v_1, f_0 \right)$ belongs to a suitable Sobolev space.

Here $0 < \tau(t)$ is the time-varying delay and $\mu_1(t)$ and $\mu_2(t)$ are nonlinear weights acting on the frictional damping. As in \cite{NICAISE_PIGNOTTI_1}, we assume that
\begin{equation}\label{hipot_1}
\tau(t) \in W^{2,\infty}([0,T]), \quad \forall T>0
\end{equation}
and that there exist positive constants $\tau_0$, $\tau_1$ and $d$ satisfying
\begin{equation}\label{hipot_2}
0 < \tau_0 \leq \tau(t) \leq \tau_1, \,\,\, \tau'(t) \leq d < 1, \quad \forall t>0.
\end{equation}

We are interested in proving the exponential stability for the problem \eqref{EQ1.1}-\eqref{EQ1.4}. In order to obtain this, we will assume that
\begin{equation}\label{assumption}
\max \{1, \frac{a}{b}\} < \frac{L_1 + L_3 - L_2}{2(L_2 - L_1)}.
\end{equation}
As described in \cite{Benseghir}, the assumption \eqref{assumption} gives the relationship between the boundary regions and the transmission permitted. It can be also seen as a restriction on the wave speeds of the two equations and the damped part of the domain. It is known that for Timoshenko systems \cite{Soufyane} and Bresse systems \cite{Boussouira_Rivera_Dilberto} the wave speeds always control the decay rate of the solution. It is an interesting open question to investigate the behavior of the solution when \eqref{assumption} is not satisfied.

Time delay is the property of a physical system by which the response to an applied force is delayed in its effect, and the central question is that delays source can destabilize a system that is asymptotically stable in the absence of delays, see \cite{DATKO,Datko2,Guesmia,Xu}.

Transmission problems are closely related to the design of material components, attracting considerable attention in recent years, e.g., in the analysis of damping mechanisms in the metallurgical industry or smart materials technology, see \cite{Balmes,Rao} and the references therein. Studies of fluid structure interaction and the added mass effect, also known as virtual mass effect, hydrodynamic mass, and hydroelastic vibration of structures, started with H. Lamb \cite{Lamb} who investigated the vibrations of a thin elastic circular plate in contact with water. Experimental study of impact on composite plates with fluid-structure interaction was investigated in \cite{Kwon}. From the mathematical point of view a transmission problem for wave propagation consists on a hyperbolic equation for which the corresponding elliptic operator has discontinuous coefficients.

We consider the wave propagation over bodies consisting of two physically different materials, one purely elastic and another subject to frictional damping. The type of wave propagation generated by mixed materials originates a transmission (or diffraction) problem.

To the best of our knowledge, the first  contribution in literature for transmission problem with a time delay
was given by  A. Benseghir in \cite{Benseghir}. More precisely,  in \cite{Benseghir} the transmission  problem
\begin{equation}\label{problem_C}
\begin{gathered}
u_{tt} - au_{xx} + \mu_1 u_t(x,t) + \mu_2 u_t(x,t-\tau) = 0, \,\, \mbox{in } \Omega \times ]0, \infty[, \\
v_{tt} - bv_{xx} = 0, \,\, \mbox{in } ]L_1, L_2[ \times ]0, \infty[.
\end{gathered}
\end{equation}
with constant weights $\mu_1, \mu_2$ and time delay $\tau > 0$ was studied. Under an appropriate assumption on the weights of the two feedbacks ($\mu_1 < \mu_2$), it was proved the well-posedness of the system and, under condition \eqref{assumption}, it was  established an exponential decay result.

The result in \cite{Benseghir} were improved by S. Zitouni et al. \cite{Zitouni_Abdelouaheb_Zennir_Rachida}. There,  the authors considered the problem with a time-varying delay $\tau(t)$ of the form
\begin{equation}\label{problem_E}
\begin{gathered}
u_{tt} - au_{xx} + \mu_1 u_t(x,t) + \mu_2 u_t(x,t-\tau(t)) = 0, \,\, \mbox{in } \Omega \times ]0, \infty[, \\
v_{tt} - bv_{xx} = 0, \,\, \mbox{in } ]L_1, L_2[ \times ]0, \infty[
\end{gathered}
\end{equation}
and proved the global existence and exponential stability under suitable assumptions on the delay term and assumption \eqref{assumption}.
Without delay, systems  \eqref{problem_C}, \eqref{problem_E} has been investigated in \cite{Bastos_Raposo}.

The transmission problem with history and delay was considered by G. Li et al., \cite{Li} where the equations were expressed as
\begin{equation}\label{problem_D}
\begin{gathered}
u_{tt} \! - \! au_{xx} \! + \!\! \int_{0}^{\infty} \!\! g(s)u_{xx}(x, t-s)ds \! + \! \mu_1 u_t(x,t) \! + \! \mu_2 u_t(x,t-\tau) = 0, \, \mbox{in } \Omega \times ]0, \infty[,\\
v_{tt} \! - \! bv_{xx} = 0, \, \mbox{in } ]L_1, L_2[ \times ]0, \infty[.
\end{gathered}
\end{equation}
Under suitable assumptions on the delay term and on the function $g$, the authors proved an exponential stability result for two cases. In the first case, they considered $\mu_2 < \mu_1$ and for second case, they assumed that $\mu_2 = \mu_1$.

S. Zitouni  et al., \cite{Zitouni_Ardjouni_Zennir_Amiar_2} extended the result in \cite{Li}  for  varying delay.
In \cite{Zitouni_Ardjouni_Zennir_Amiar_2} was proved existence and the uniqueness of the solution by using the semigroup theory and the exponential stability of the solution by the energy method for the following problem
\begin{equation}\label{problem_D_1}
\begin{gathered}
u_{tt} \! - \! au_{xx} \! + \!\! \int_{0}^{\infty} \!\!\! g(s)u_{xx}(x, t \! - \! s)ds \! + \! \mu_1 u_t(x,t) \! + \! \mu_2 u_t(x,t-\tau(t)) \! = \! 0, \, \mbox{in } \Omega \times ]0, \infty[,\\
v_{tt} \! - \! bv_{xx} \! = \! 0, \, \mbox{in } ]L_1, L_2[ \times ]0, \infty[.
\end{gathered}
\end{equation}

Stability to localized viscoelastic transmission problem was considered by Mu\~noz Rivera et al., \cite{Rivera_Octavio_Mauricio} where they considered
\begin{equation}\label{problem_D2}
\begin{gathered}
\rho \phi_{tt} + \sigma_{x} = 0, \\
\sigma(x,t) = \alpha(x)\varphi_{x} - k(x)\varphi_{xt} - \beta(x)\varphi_{xt} = 0.
\end{gathered}
\end{equation}
In \cite{Rivera_Octavio_Mauricio} the authors investigated the effect of the positions of the dissipative mechanisms on a bar with three component $]0,L_0[,\, ]L_0,L_1[, \, ]L_1,L[$, and showed that the system is exponentially stable if and only if the viscous component is not in the center of the bar. In other case, they showed the lack of exponential stability, and that the solutions still decay but just polynomially to zero.

The case of time-varying delay has already been considered in other works, such as  \cite{Orlov_Fridman,Kirane_Said_Anwar,Liu,Zitouni_Ardjouni_Zennir_Amiar}. Wave equations with time-varying delay and nonlinear weights was considered in the recent work of Barros et al., \cite{Vanessa} where was studied the equation given by
\begin{equation}\label{NLS}
u_{tt} - u_{xx} + \mu_1(t)u_t +\mu_2(t)u_t(x,t-\tau(t)) = 0, \,\, \mbox{in } ]0,L[ \times ]0, +\infty[.
\end{equation}
Under proper conditions on nonlinear weights $\mu_1(t), \mu_2(t)$ and  $\tau(t)$, authors proved global existence and an estimate for the decay rate of the energy.

In the present work we improve the results in \cite{Zitouni_Abdelouaheb_Zennir_Rachida} where, for constant weights $\mu_1(t)=\mu_1$,   $\mu_2(t)=\mu_2$ and under adequate assumptions regarding the weight and time-varying delay, was proved the well posedness and singularity of solutions by using the semigroup theory. Authors also showed exponential stability by introducing an appropriate Lyapunov functional.

Here we consider a transmission problem with nonlinear weights and time-varying delay, which is the main characteristic of this work. Although there are some works on laminated beam and on Timoshenko system with delay, all of them consider constant weights, i.e., $\mu_1$ and $\mu_2$ are constants. To the best of our knowledge, there is no result for these systems with nonlinear weights. Moreover, since the weights are nonlinear, a difficulty comes in: the operator is nonautonomous. This makes hard the use semigroup theory to study well-posedness. To overcome it we use the Kato's variable norm technique together with semigroup theory to show that the system is well-posed.

The remainder of this paper is organized as follows. In section \ref{sec:Notation_preliminaries} we introduce some notations and prove the dissipative property for the energy of the system. In the section \ref{sec:global_solution}, by using Kato's variable norm technique and under some restriction on the non-linear weights and the time-varying delay, the system is shown to be well-posed. In section \ref{sec:Exponential_stability}, we present the result of exponential stability by energy methods, and by using suitable sophisticated estimates for multipliers to construct an appropriated Lyapunov functional.

\section{Notation and preliminaries}\label{sec:Notation_preliminaries}

We start by setting the following hypothesis:\\
{\bf (H1)}  $\mu_1:\mathbb{R}_+ \rightarrow ]0,+\infty[$ is a non-increasing function of class $C^1(\mathbb{R}_+)$ satisfying
\begin{equation}\label{H1}
\left| \frac{\mu'_1(t)}{\mu_1(t)}\right| \leq M_1, \quad \forall t \geq 0,
\end{equation}
where $M_1 > 0$ is a constant.

{\bf (H2)}  $\mu_2:\mathbb{R}_+ \rightarrow \mathbb{R}$ is a function of class $C^1(\mathbb{R}_+)$, which is not necessarily positive or monotone, such that
\begin{gather}\label{H2_1}
\left| \mu_2(t) \right| \leq \beta \mu_1(t), \\
\left| \mu'_2(t) \right| \leq M_2 \mu_1(t),
	\end{gather}
for some $0 < \beta < \sqrt{1-d}$ and $M_2>0$.

As in Nicaise and Pignotti \cite{NICAISE_PIGNOTTI_1} we introduce the new variable
\begin{equation}\label{EQ2.1}
z(x,\rho,t) = u_t(x,t - \tau(t) \rho),\quad (x,\rho) \in \Omega \times ]0,1[, \; t>0.
\end{equation}

It is easily verified that the new variable satisfies
$$\tau(t)z_t(x,\rho,t) + (1 - \tau'(t)\rho)z_\rho(x,\rho,t) = 0$$
and the problem \eqref{EQ1.1} is equivalent to

\begin{equation}\label{EQ2.2}
\begin{gathered}
u_{tt}(x,t) - au_{xx}(x,t) + \mu_1(t)u_t(x,t) + \mu_2(t)z(x,1,t) = 0 \quad \mbox{in } \Omega \times ]0, \infty[, \\
v_{tt}(x,t) - bv_{xx}(x,t) = 0 \quad \mbox{in } ]L_1, L_2[ \times ]0, \infty[, \\
\tau(t)z_t(x,\rho,t) + (1 - \tau'(t)\rho)z_\rho(x,\rho,t) = 0 \quad \mbox{in } \Omega \times ]0,1[ \times ]0, \infty[.
\end{gathered}
\end{equation}
This system is subject to the transmission conditions
\begin{equation}
\begin{gathered}
u(L_i, t) = v(L_i, t), \quad i= 1,2, \\
au_x(L_i, t) = bv_x(L_i, t), \quad i= 1,2,
\end{gathered} \label{EQ2.2.1}
\end{equation}
the boundary conditions
\begin{equation}\label{EQ2.2.2}
u(0, t) = u(L_3, t) = 0
\end{equation}
and the initial conditions
\begin{equation}\label{EQ2.3}
\begin{gathered}
u(x,0) = u_0(x), \quad u_t(x,0) = u_1(x) \quad \mbox{on } \Omega, \\
v(x,0) = v_0(x), \quad v_t(x,0) = v_1(x) \quad \mbox{on } \Omega, \\
z(x,\rho,0) = u_t(x, -\tau(0)\rho) = f_0(x, -\tau(0)\rho), \quad (x,\rho) \quad \mbox{in } \Omega \times ]0, 1[.
\end{gathered}
\end{equation}

For any regular solution of \eqref{EQ2.2}, we define the energy as
\begin{gather*}
E_1(t) = \frac{1}{2} \int_\Omega \left( |u_t(x,t)|^2 + a|u_x(x,t)|^2 \right)\,dx, \\
E_2(t) = \frac{1}{2} \int_{L_1}^{L_2}\left( |v_t(x,t)|^2 + b|v_x(x,t)|^2 \right)\,dx.
\end{gather*}
The total energy is defined by
\begin{equation}\label{EQ2.8}
E(t) = E_1(t) + E_2(t) + \frac{\xi(t)\tau(t)}{2} \int_\Omega \int_{0}^{1} z^2(x,\rho,t)\,d\rho\,dx,
\end{equation}
where
\begin{equation}\label{hipotese_7}
\xi(t)=\bar{\xi}\mu_1(t)
\end{equation}
is a non-increasing function of class $C^1(\mathbb{R}_+)$ and $\bar{\xi}$ be a positive constant such that
\begin{equation}\label{hipotese_8}
\frac{\beta}{\sqrt{1-d}} < \bar{\xi} < 2 - \frac{\beta}{\sqrt{1-d}}.
\end{equation}

Our first result states that the energy is a non-increasing function.

\begin{lemma}\label{Lemma_2.2}
Let $(u,v,z)$ be a solution to the system \eqref{EQ2.2}-\eqref{EQ2.3}. Then the energy functional defined by \eqref{EQ2.8} satisfies
\begin{align}\label{derivate_energy}
E'(t) & \leq -\mu_1(t) \left( 1-\frac{\bar{\xi}}{2}- \frac{\beta}{2\sqrt{1-d}} \right) \int_\Omega u_t^2\,dx \\
      & \quad - \mu_1(t) \left( \frac{\bar{\xi}(1-\tau'(t))}{2}- \frac{\beta \sqrt{1-d}}{2} \right) \int_\Omega z_1^2(x,\rho,t)\,dx \nonumber\\
      & \leq 0 \nonumber.
\end{align}
\end{lemma}

\begin{proof}
Multiplying the first and second equations of \eqref{EQ2.2} by $u_t(x,t)$ and $v_t(x,t)$, integrating on $\Omega$ and $]L_1,L_2[$ respectively and using integration by parts, we get

\begin{align}\label{EQ2.12}
\frac{1}{2}\frac{d}{dt} \int_\Omega \left( u_t^2 + au_x^2 \right)\,dx = & -\mu_1(t) \int_\Omega u_t^2\,dx - \mu_2(t)\int_\Omega z(x,1,t) u_t\,dx + a \left[ u_x u_t \right]_{\partial \Omega},
\end{align}
\begin{equation}
\frac{1}{2} \frac{d}{dt} \int_{L_1}^{L_2}\left( v_t^2 + bv_x^2 \right)\,dx = b \left[ v_x v_t \right]_{L_1}^{L_2}. \label{EQ2.13}
\end{equation}

Now multiplying the third equation of \eqref{EQ2.3} by $\xi(t)z(x,\rho,t)$ and integrating on $\Omega \times ]0,1[$, we obtain
\begin{align*}
&\tau(t)\xi(t)\int_{\Omega} \int_0^1 z_t(x,\rho,t)z(x,\rho,t)\,d\rho\,dx = -\frac{\xi(t)}{2} \int_{\Omega} \int_0^1 (1- \tau'(t)\rho)\frac{\partial}{\partial \rho}(z(x,\rho,t))^2\,d\rho\,dx.
\end{align*}
Consequently,
\begin{align}\label{equ1}
\frac{d}{dt} \left( \frac{\xi(t)\tau(t)}{2} \int_{\Omega} \int_0^1 z^2(x,\rho,t)\,d\rho\,dx \right) & = \frac{\xi(t)}{2} \int_{\Omega} (z^2(x,0,t)-z^2(x,1,t))\,dx \\
&\quad + \frac{\xi(t)\tau'(t)}{2} \int_{\Omega} \int_0^1 z^2(x,1,t)\,d\rho\,dx \nonumber \\
&\quad + \frac{\xi'(t)\tau(t)}{2} \int_{\Omega} \int_0^1 z^2(x,\rho,t)\,d\rho\,dx. \nonumber
\end{align}
From \eqref{EQ2.8}, \eqref{EQ2.12}, \eqref{EQ2.13}, \eqref{equ1} and using the conditions \eqref{EQ2.2.1} and \eqref{EQ2.2.2}, we know that
\begin{align}\label{equ2}
E'(t) &= \frac{\xi(t)}{2} \int_{\Omega} u_t^2\,dx - \frac{\xi(t)}{2} \int_{\Omega} z^2(x,1,t)\,dx + \frac{\xi(t)\tau'(t)}{2} \int_{\Omega} z^2(x,1,t)\,dx \\
& \quad + \frac{\xi'(t)\tau(t)}{2} \int_{\Omega} \int_0^1 z^2(x,\rho,t)\,d\rho\,dx - \mu_1(t)\int_{\Omega} u_t^2\,dx - \mu_2(t)\int_{\Omega} z(x,1,t) u_t\,dx. \nonumber
\end{align}
Due to Young's inequality, we have
\begin{align}\label{equ3}
\mu_2(t) \int_{\Omega} z(x,1,t) u_t\,dx \leq & \frac{\left| \mu_2(t) \right|}{2\sqrt{1-d}} \int_{\Omega}u_t^2\,dx + \frac{\left| \mu_2(t) \right| \sqrt{1-d}}{2} \int_{\Omega} z^2(x,1,t)\,dx.
\end{align}
Inserting \eqref{equ3} into \eqref{equ2}, we obtain

\begin{align*}
E'(t) &\leq -\left( \mu_1(t) - \frac{\xi(t)}{2} - \frac{\left| \mu_2(t) \right|}{2\sqrt{1-d}} \right) \int_{\Omega} u_t^2\,dx \\
&\;\;\;\; - \left( \frac{\xi(t)}{2} - \frac{\xi(t)\tau'(t)}{2} - \frac{\left| \mu_2(t) \right| \sqrt{1-d}}{2} \right) \int_{\Omega} z^2(x,1,t)\,dx \\
&\;\;\;\; + \frac{\xi'(t)\tau(t)}{2} \int_{\Omega} \int_0^1 z^2(x,\rho,t)\,d\rho\,dx \\
&\leq -\mu_1(t) \left( 1-\frac{\bar{\xi}}{2}- \frac{\beta}{2\sqrt{1-d}} \right) \int_{\Omega} u_t^2\,dx \\
& \;\;\;\; - \mu_1(t) \left( \frac{\bar{\xi}(1-\tau'(t))}{2}- \frac{\beta \sqrt{1-d}}{2} \right) \int_{\Omega} z^2(x,1,t)\,dx \\
& \leq 0.
\end{align*}
Hence, the proof is complete.
\end{proof}

\section{Global solution}\label{sec:global_solution}

In this section, our goal is to prove existence and uniqueness of solutions to the system \eqref{EQ2.2} - \eqref{EQ2.3}. This is the content of Theorem \ref{Global_Solution}.

We begin by introducing the vector function $U = (u,v,\varphi,\psi,z)^T$, where $\varphi(x,t) = u_t(x,t)$ and $\psi(x,t) = v_t(x,t)$. The system \eqref{EQ2.2}-\eqref{EQ2.3} can be written as
\begin{equation}\label{EQ2.17}
\left\{\begin{array}{ll}
U_{t} - \mathcal{A}(t) U= 0, \\
U(0) = U_0 = (u_0,v_0,u_1,v_1,f_0(\cdot,-,\tau(0)))^T,
\end{array}
\right.\quad
\end{equation}
where the operator $\mathcal{A}(t)$ is defined by
\begin{equation}\label{EQ2.18}
\mathcal{A}(t)\,U = \left(
\begin{array}{c}
\varphi(x,t) \\
\psi(x,t) \\
au_{xx}(x,t) - \mu_1(t)\varphi(x,t) - \mu_2(t)z(x,1,t) \\
bv_{xx}(x,t) \\
- \frac{1 - \tau'(t)\rho}{\tau(t)} z_{\rho}(x,\rho,t)
\end{array}
\right).
\end{equation}

Now, taking into account the conditions \eqref{EQ1.2}-\eqref{EQ1.3}, as well as previous results presented in \cite{Benseghir,LiuG,Li,Zitouni_Abdelouaheb_Zennir_Rachida}, we introduce the set
\begin{align*}
X_{*} = & \{ (u,v) \in H^1(\Omega) \times H^1(]L_1,L_2[)/ u(0)=u(L_3)=0, u(L_i)=v(L_i), au_x(L_i) = bv_x(L_i), i= 1,2 \}.
\end{align*}
We define the phase space as
$$\mathcal{H} = X_{*} \times L^2(\Omega) \times L^2(]L_1,L_2[) \times L^2(\Omega \times ]0,1[)$$
equipped with the inner product
\begin{align}\label{inner_product}
\langle U,\hat{U} \rangle_{\mathcal{H}} =  \int_{\Omega} \left( \varphi \hat{\varphi} + au_x \hat{u}_x \right)\,dx + \int_{L_1}^{L_2} \left( \psi \hat{\psi} + bv_x \hat{v}_x \right)\,dx  + \xi(t)\tau(t) \int_{\Omega} \int_{0}^{1} z \hat{z}\,d\rho\,dx,  
\end{align}
for $U = (u,v,\varphi,\psi,z)^T$ and $\hat{U} = (\hat{u},\hat{v},\hat{\varphi},\hat{\psi},\hat{z})^T$.

The domain $D(\mathcal{A}(t))$ of $\mathcal{A}(t)$ is defined by
\begin{align}\label{domain}
D(\mathcal{A}(t)) = \{ & (u,v,\varphi,\psi,z)^T \in \mathcal{H}/ (u,v) \in \left( H^2(\Omega) \times H^2(]L_1,L_2[) \right) \cap X_{*}, \\
& \varphi \in H^1(\Omega), \psi \in H^1(]L_1,L_2[), z \in L^2 \left( ]0,L[; H_0^1(]0,1[) \right), \varphi=z(\cdot,0) \}. \nonumber
\end{align}

Notice that the domain of the operator $\mathcal{A}(t)$ does not dependent on time $t$, i.e.,
\begin{equation}\label{DA(t)=DA(0)}
D(\mathcal{A}(t)) = D(\mathcal{A}(0)), \quad \forall t>0.
\end{equation}

A general theory for not autonomous operators given by equations of type \eqref{EQ2.17} has been developed using semigroup theory, see \cite{Kato_1}, \cite{Kato_3} and \cite{Pazy}. The simplest way to prove existence and uniqueness results is to show that the triplet $\left\{ (\mathcal{A}, \mathcal{H}, Y) \right\}$, with $\mathcal{A} = \left\{ \mathcal{A}(t)/ t \in [0,T] \right\}$, for some fixed $T>0$ and $Y=\mathcal{A}(0)$, forms a CD-systems (or constant domain system, see \cite{Kato_1} and \cite{Kato_3}). More precisely, the following theorem, which id due to Tosio Kato (Theorem 1:9 of \cite{Kato_1}) gives the existence and uniqueness results and is proved in Theorem $1.9$ of \cite{Kato_1} (see also Theorem $2.13$ of \cite{Kato_3} or \cite{Ali}). For convenience let states Kato's result here.

\begin{theorem}\label{Theorem_preliminar}
Assume that
\begin{enumerate}
\item[(i)] $Y=D(\mathcal{A}(0))$ is  dense subset of $\mathcal{H}$,
\item[(ii)] \eqref{DA(t)=DA(0)} holds,
\item[(iii)] for all $t \in [0,T]$, $\mathcal{A}(t)$ generates a strongly continuous semigroup on $\mathcal{H}$ and the family $\mathcal{A}(t) = \left\{ \mathcal{A}(t)/ t \in [0,T] \right\}$ is stable with stability constants $C$ and $m$ independent of $t$ (i.e., the semigroup $(S_t(s))_{s\geq 0}$ generated by $\mathcal{A}(t)$ satisfies $\| S_t(s)u \|_{\mathcal{H}} \leq Ce^{ms} \| u \|_{\mathcal{H}}$, for all $u \in \mathcal{H}$ and $s\geq 0$),
\item[(iv)] $\partial_t \mathcal{A}(t)$ belongs to $L_{*}^{\infty}([0,T],B(Y, \mathcal{H}))$, which is the space of equivalent classes of essentially bounded, strongly measurable functions from $[0,T]$ into the set $B(Y, \mathcal{H})$ of bounded linear operators from $Y$ into $\mathcal{H}$.

Then, problem \eqref{EQ2.17} has a unique solution $U \in C([0,T],Y) \cap C^1([0,T], \mathcal{H})$ for any initial datum in $Y$.
\end{enumerate}
\end{theorem}

Using the time-dependent inner product \eqref{inner_product} and the Theorem \ref{Theorem_preliminar}  we get the following result
of existence and uniqueness of global solutions to the problem \eqref{EQ2.17}.

\begin{theorem}\label{Global_Solution}[Global solution]
For any initial datum $U_0 \in \mathcal{H}$ there exists a unique solution $U$ satisfying
$$
U \in C([0,+\infty[, \mathcal{H})
$$
for problem \eqref{EQ2.17}.

Moreover, if $U_0 \in D(\mathcal{A}(0))$, then
$$
U \in C([0,+\infty[, D(\mathcal{A}(0))) \cap C^1([0,+\infty[, \mathcal{H}).
$$
\end{theorem}

\begin{proof}
Our goal is then to check the above assumptions for problem \eqref{EQ2.17}.
First, we show that $D(\mathcal{A}(0))$ is dense in $\mathcal{H}$.
The proof we will follow method used in \cite{Nicaise_Pignotti} with the necessary modification imposed by the nature of our problem.
Let $\hat{U} = (\hat{u},\hat{v},\hat{\varphi},\hat{\psi},\hat{z})^T \in \mathcal{H}$ be orthogonal to all elements of $D(\mathcal{A}(0))$, namely

\begin{align}\label{x_1}
0 = \langle U,\hat{U} \rangle_{\mathcal{H}} = & \int_{\Omega} \left( \varphi \hat{\varphi} + au_x \hat{u}_x \right)\,dx + \int_{L_1}^{L_2} \left( \psi \hat{\psi} + bv_x \hat{v}_x \right)\,dx + \xi(t)\tau(t) \int_{\Omega} \int_{0}^{1} z \hat{z}\,d\rho\,dx, 
\end{align}
for $U = (u,v,\varphi,\psi,z)^T \in D(\mathcal{A}(0))$.

We first take $u=v=\varphi=\psi=0$ and $z \in C_{0}^{\infty}\left(\Omega \times ]0,1[ \right)$. As $U=(0,0,0,0,z)^T \in D(\mathcal{A}(0))$ and therefore, from \eqref{x_1}, we deduce that
$$
\int_{\Omega} \int_0^1 z\hat{z}\,d\rho\,dx = 0.
$$
Since $C_{0}^{\infty}\left(\Omega \times ]0,1[ \right)$ is dense in $L^2\left(\Omega \times ]0,1[ \right)$, then, it follows then that $\hat{z}=0$. Similarly, let $\varphi \in C_{0}^{\infty}(\Omega)$, then $U=(0,0,\varphi,0,0)^T \in D(\mathcal{A}(0))$, which implies from \eqref{x_1} that
$$
\int_{\Omega} \varphi \hat{\varphi}\,dx  = 0.
$$
So, as above, it follows that $\hat{\varphi} = 0$. In the same way, by taking $\psi \in C_{0}^{\infty}(]L_1,L_2[)$, we get from \eqref{x_1}
$$
\int_{L_1}^{L_2} \psi \hat{\psi}\,dx  = 0
$$
and by density of $C_{0}^{\infty}(]L_1,L_2[)$ in $L^2(]L_1,L_2[)$, we obtain $\hat{\psi} = 0$.

Finally, for $(u,v) \in C_{0}^{\infty}\left(\Omega \times ]L_1,L_2[ \right)$ (then $(u_x,v_x) \in C_{0}^{\infty}\left(\Omega \times ]L_1,L_2[ \right)$) we obtain
$$
a\int_{\Omega} u_x\hat{u}_x\,dx + b\int_{L_1}^{L_2} v_x\hat{v}_x\,dx = 0.
$$
Since $C_{0}^{\infty} \left(\Omega \times ]L_1,L_2[ \right)$ is dense in $L^2 \left( \Omega \times ]L_1,L_2[ \right)$, we deduce that $\left( \hat{u}_x,\hat{v}_x \right)=(0,0)$ because $\left( \hat{u},\hat{v} \right) \in X_{*}$.

We consequently have
\begin{equation}\label{x_2}
D(\mathcal{A}(0)) \,\, \text{is dense in } \mathcal{H}.
\end{equation}

Now, we show that the operator $\mathcal{A}(t)$ generates a $C_0-$semigroup in $\mathcal{H}$ for a fixed $t$.

We calculate $\langle \mathcal{A}(t)U, U \rangle_t$ for a fixed $t$. Take $U=(u,v,\varphi,\psi,z)^T \in D(\mathcal{A}(t))$. Then
\begin{align*}
\langle \mathcal{A}(t)U, U \rangle_t = & - \mu_1(t) \int_{\Omega} \varphi^2\,dx - \mu_2(t) \int_{\Omega} z(x,1)\varphi\,dx - \frac{\xi(t)}{2} \int_{\Omega} \int_0^1 (1-\tau'(t)\rho)\dfrac{\partial}{\partial \rho} z^2(x,\rho)\,d\rho\,dx.
\end{align*}
Since
$$
\left(1- \tau'(t)\rho \right)\dfrac{\partial}{\partial \rho} z^2(x,\rho) = \dfrac{\partial}{\partial \rho} \left( \left(1- \tau'(t)\rho \right) z^2(x,\rho) \right) + \tau'(t)z^2(x,\rho),
$$
we have
\begin{align*}
\int_0^1 \left(1- \tau'(t)\rho \right)\dfrac{\partial}{\partial \rho} z^2(x,\rho)\,d\rho = & \left( 1- \tau'(t) \right)z^2(x,1) - z^2(x,0) + \tau'(t)\int_0^1 z^2(x,\rho)\,d\rho.
\end{align*}
Whereupon
\begin{align*}
\langle \mathcal{A}(t)U, U \rangle_t =& - \mu_1(t) \int_{\Omega} \varphi^2\,dx - \mu_2(t) \int_{\Omega} z(x,1)\varphi\,dx + \frac{\xi(t)}{2} \int_{\Omega} \varphi^2\,dx \\
& - \frac{\xi(t)(1-\tau'(t))}{2} \int_{\Omega} z^2(x,1)\,dx - \frac{\xi(t)\tau'(t)}{2} \int_{\Omega} \int_0^1 z^2(x,\rho)\,d\rho\,dx.
\end{align*}

Therefore, from \eqref{equ3}, we deduce

\begin{align*}
\langle \mathcal{A}(t)U, U \rangle_t \leq &- \mu_1(t) \left( 1-\frac{\bar{\xi}}{2}- \frac{\beta}{2\sqrt{1-d}} \right) \int_{\Omega} \varphi^2\,dx \\
&- \mu_1(t) \left( \frac{\bar{\xi}(1-\tau'(t))}{2}- \frac{\beta \sqrt{1-d}}{2} \right)\int_{\Omega} z^2(x,1)\,dx \\
&+ \frac{\xi(t)|\tau'(t)|}{2\tau(t)}\tau(t) \int_{\Omega}\int_0^1 z^2(x,\rho)\,d\rho\,dx.
\end{align*}
Then, we have

\begin{align*}
\langle \mathcal{A}(t)U, U \rangle_t \leq &- \mu_1(t) \left( 1-\frac{\bar{\xi}}{2}- \frac{\beta}{2\sqrt{1-d}} \right) \int_{\Omega} \varphi^2\,dx \\
&- \mu_1(t) \left( \frac{\bar{\xi}(1-\tau'(t))}{2}- \frac{\beta \sqrt{1-d}}{2} \right)\int_{\Omega} z^2(x,1)\,dx \\
&+ \kappa(t) \langle U, U \rangle_t,
\end{align*}
where
$$
\kappa(t) = \frac{\sqrt{1+\tau'(t)^2}}{2\tau(t)}.
$$

From \eqref{derivate_energy} we conclude that
\begin{equation}\label{dissip}
\langle \mathcal{A}(t)U, U \rangle_t - \kappa(t)\langle U,U \rangle_t \leq 0,
\end{equation}
which means that operator $\tilde{\mathcal{A}}(t) = \mathcal{A}(t) - \kappa(t) I$ is dissipative.

Now, we prove  the surjectivity of the operator $\lambda I - \mathcal{A}(t)$ for fixed $t>0$ and $\lambda >0$. For this purpose, let $F=(f_1,f_2,f_3,f_4,f_5)^T \in \mathcal{H}$. We seek $U=(u,v,\varphi,\psi,z)^T \in D(\mathcal{A}(t))$ which is solution of
$$
\left( \lambda I - \mathcal{A}(t) \right)U=F,
$$
that is, the entries of $U$ satisfy the system of equations
\begin{eqnarray}
\lambda u - \varphi = f_1, \label{q_1}\\
\lambda v - \psi = f_2, \label{q_2}\\
\lambda \varphi - au_{xx} + \mu_1(t)\varphi + \mu_2(t) z(x,1) = f_3, \label{q_3}\\
\lambda \psi - bv_{xx} = f_4, \label{q_4}\\
\lambda z + \frac{1 - \tau'(t)\rho}{\tau(t)}z_{\rho} = f_5. \label{q_5}
\end{eqnarray}

Suppose that we have found $u$ and $v$ with the appropriated regularity. Therefore, from \eqref{q_1} and \eqref{q_2} we have
\begin{eqnarray}
\varphi = \lambda u - f_1, \label{x_6.1} \\
\psi = \lambda v - f_2. \label{x_6.2}
\end{eqnarray}
It is clear that $\varphi \in H^1(\Omega)$ and $\psi \in H^1(]L_1,L_2[)$. Furthermore, if $\tau'(t) \neq 0$, following the same approach as in \cite{Nicaise_Pignotti}, we obtain
$$
z(x,\rho) = \varphi(x)e^{\sigma(\rho,t)} + \tau(t) e^{\sigma(\rho,t)} \int_{0}^{\rho} \frac{f_5(x,s)}{1-s\tau'(s)} e^{-\sigma(s,t)}\,ds,
$$
where
$$\sigma(\rho,t) = \lambda \frac{\tau(t)}{\tau'(t)}\ln(1- \rho \tau'(t)),$$
is solution of \eqref{q_5} satisfying
\begin{equation}\label{t_1}
z(x,0)=\varphi(x).
\end{equation}
Otherwise,
$$
z(x,\rho) = \varphi(x)e^{-\lambda \tau(t) \rho} + \tau(t)e^{-\lambda \tau(t) \rho} \int_{0}^{\rho} f_5(x,s)e^{\lambda \tau(t) s}\,ds
$$
is solution of \eqref{q_5} satisfying \eqref{t_1}. From now on, for pratical purposes we will consider $\tau'(t) \neq 0$ (the case $\tau'(t) = 0$ is analogous). Taking into account \eqref{x_6.1} we have
\begin{align}\label{x_7}
z(x,1) & = \varphi e^{\sigma(1,t)} + \tau(t) e^{\sigma(1,t)} \int_{0}^{1} \frac{f_5(x,s)}{1-s\tau'(s)} e^{-\sigma(s,t)}\,ds \\
& = \left( \lambda u - f_1 \right) e^{\sigma(1,t)} + \tau(t) e^{\sigma(1,t)} \int_{0}^{1} \frac{f_5(x,s)}{1-s\tau'(s)} e^{-\sigma(s,t)}\,ds \nonumber \\
& = \lambda u e^{\sigma(1,t)} - f_1 e^{\sigma(1,t)} + \tau(t) e^{\sigma(1,t)} \int_{0}^{1} \frac{f_5(x,s)}{1-s\tau'(s)} e^{-\sigma(s,t)}\,ds. \nonumber
\end{align}
Substituting \eqref{x_6.1} and \eqref{x_7} in \eqref{q_3}, and \eqref{x_6.2} in \eqref{q_4}, we obtain
\begin{equation}\label{x_11}
\left\{ \begin{array}{l}
\eta u - au_{xx} = g_1, \\
\lambda^2 v - b v_{xx} = g_2,
\end{array}\right.
\end{equation}
where
\begin{equation*}
\eta := \lambda^2 + \lambda \mu_1(t) + \lambda \mu_2(t) e^{\sigma(1,t)},
\end{equation*}
\begin{align*}
g_1 := & f_3 + \lambda f_1 + \mu_1(t) f_1 + \mu_2(t) f_1 e^{\sigma(1,t)} \\
& - \mu_2(t) \tau(t) e^{\sigma(1,t)} \int_{0}^{1} \frac{f_5(x,s)}{1-s\tau'(s)} e^{-\sigma(s,t)}\,ds,
\end{align*}
\begin{equation*}
g_2:= f_4 + \lambda f_2.
\end{equation*}
In order to solve \eqref{x_11}, we use a standard procedure, considering variational problem
\begin{equation}\label{variational_problem}
\Upsilon( (u,v), (\tilde{u}, \tilde{v}) ) = L(\tilde{u},\tilde{v}),
\end{equation}
where the bilinear form
\begin{eqnarray*}
\Upsilon: X_{*} \times X_{*} \rightarrow \mathbb{R}
\end{eqnarray*}
and the linear form
$$
L: X_{*} \rightarrow \mathbb{R}
$$
are defined by
\begin{align*}
\Upsilon( (u,v), (\tilde{u},\tilde{v}) ) = & \eta \int_{\Omega} u \tilde{u}\,dx + a \int_{\Omega} u_x \tilde{u}_x\,dx + \lambda^2 \int_{L_1}^{L_2} v \tilde{v}\,dx + b \int_{L_1}^{L_2} v_x \tilde{v}_x\,dx - a \left[ u_x \tilde{u} \right]_{\partial \Omega} - b \left[ v_x \tilde{v} \right]_{L_1}^{L_2}
\end{align*}
and
\begin{equation*}
L(\tilde{u},\tilde{v}) = \int_{\Omega} g_1 \tilde{u}\,dx + \int_{L_1}^{L_2} g_2 \tilde{v}\,dx.
\end{equation*}

It is easy to verify that $\Upsilon$ is continuous and coercive, and $L$ is continuous, so by applying the Lax-Milgram Theorem, we obtain a solution for $(u,v) \in X_{*}$ for \eqref{x_11}. In addition, it follows from \eqref{q_3} and \eqref{q_4} that $(u,v) \in H^2(\Omega) \times H^2(]L_1,L_2[)$ and so $(u,v,\varphi,\psi,z) \in D(\mathcal{A}(t))$.

Therefore, the operator $\lambda I - \mathcal{A}(t)$ is surjective for any $\lambda > 0$ and $t>0$. Again as $\kappa(t)>0$, this prove that
\begin{equation}\label{lambdaI-tildeA_surjective}
\lambda I - \tilde{\mathcal{A}}(t) = \left( \lambda + \kappa(t) \right)I - \mathcal{A}(t) \ \text{is surjective}
\end{equation}
for any $\lambda >0$ and $t>0$.

To complete the proof of (iii), it suffices to show that
\begin{equation}\label{norma}
\dfrac{\| \Phi \|_t}{\| \Phi \|_s} \leq e^{\frac{c}{2\tau_0}|t-s|}, \quad  t,s \in [0,T],
\end{equation}
where $\Phi = (u,v,\varphi,\psi,z)^T$, $c$ is a positive constant and $\| \cdot \|$ is the norm associated the inner product \eqref{inner_product}. For all $t,s \in [0,T]$, we have
\begin{align*}
\| \Phi \|_t^2 - \| \Phi \|_s^2 e^{\frac{c}{\tau_0}|t-s|} = & \left(1 - e^{\frac{c}{\tau_0}|t-s|} \right) \left[ \int_{\Omega} \left( \varphi^2 + au_x^2 \right)dx + \int_{L_1}^{L_2} \left( \psi^2 + bv_x^2 \right)dx \right] \\
&+ \left( \xi(t)\tau(t) - \xi(s)\tau(s)e^{\frac{c}{\tau_0}|t-s|} \right) \int_{\Omega} \int_0^1 z^2(x,\rho)\,d\rho\,dx.
\end{align*}
It is clear that $1 - e^{\frac{c}{\tau_0}|t-s|} \leq 0$. Now we will prove $\xi(t)\tau(t) - \xi(s)\tau(s)e^{\frac{c}{\tau_0}|t-s|} \leq 0$ for some $c>0$. In order to do this , first observe that
$$
\tau(t) = \tau(s) + \tau'(r)(t-s),
$$
for some $r \in (s,t)$. Since $\xi$ is a non increasing function and $\xi>0$, we get
$$
\xi(t)\tau(t) \leq \xi(s)\tau(s) + \xi(s)\tau'(r)(t-s),
$$
which implies
$$
\dfrac{\xi(t)\tau(t)}{\xi(s)\tau(s)} \leq 1 + \dfrac{|\tau'(r)|}{\tau(s)}|t-s|.
$$
Using \eqref{hipot_1} and that $\tau'$ is bounded, we deduce
$$
\frac{\xi(t)\tau(t)}{\xi(s)\tau(s)} \leq 1 + \frac{c}{\tau_0}|t-s| \leq e^{\frac{c}{\tau_0}|t-s|},
$$
which proves \eqref{norma} and therefore $(iii)$ follows.

Moreover, as $\kappa'(t) = \frac{\tau'(t)\tau''(t)}{2\tau(t)\sqrt{1+\tau'(t)^2}} - \frac{\tau'(t)\sqrt{1+\tau'(t)^2}}{2\tau(t)^2}$ is bounded on $[0,T]$ for all $T>0$ (by \eqref{hipot_1} and \eqref{hipotese_8}) we have
$$
\frac{d}{dt}\mathcal{A}(t)U = \left( \begin{array}{c}
0 \\
0 \\
-\mu_1'(t) \varphi - \mu_2'(t)z(\cdot,1) \\
0 \\
\frac{\tau''(t)\tau(t)\rho-\tau'(t)(\tau'(t)\rho-1)}{\tau(t)^2}z_{\rho}
\end{array} \right),
$$
with $\frac{\tau''(t)\tau(t)\rho-\tau'(t)(\tau'(t)\rho-1)}{\tau(t)^2}$ bounded on $[0,T]$ by \eqref{hipot_1} and \eqref{hipotese_8}. Thus
\begin{equation}\label{derivate_tilde_A}
\frac{d}{dt}\tilde{\mathcal{A}}(t) \in L_{*}^{\infty}([0,T], B(D(\mathcal{A}(0)), \mathcal{H})),
\end{equation}
where $L_{*}^{\infty}([0,T], B(D(\mathcal{A}(0)), \mathcal{H}))$ is the space of equivalence classes of essentially bounded, strongly measurable functions from $[0,T]$ into $B(D(\mathcal{A}(0)), \mathcal{H})$. Here $B(D(\mathcal{A}(0)), \mathcal{H})$ is the set of bounded linear operators from $D(\mathcal{A}(0))$ into $\mathcal{H}$.

Then, \eqref{dissip}, \eqref{lambdaI-tildeA_surjective} and \eqref{norma} imply that the family $\tilde{\mathcal{A}} = \left\{ \tilde{\mathcal{A}}(t): t \in [0,T] \right\}$ is a stable family of generators in $\mathcal{H}$ with stability constants independent of $t$, by Proposition $1.1$ from \cite{Kato_1}. Therefore, the assumptions $(i)-(iv)$ of Theorem \ref{Theorem_preliminar} are verified by \eqref{DA(t)=DA(0)}, \eqref{x_2}, \eqref{dissip}, \eqref{lambdaI-tildeA_surjective}, \eqref{norma} and \eqref{derivate_tilde_A}. Thus, the problem
\begin{equation}
\left\{\begin{array}{ll}
\tilde{U}_{t} = \tilde{\mathcal{A}}(t) \tilde{U},\\
\tilde{U}(0) = U_0
\end{array}
\right.\quad
\end{equation}
has a unique solution $\tilde{U} \in C\left( [0,+\infty[, D(\mathcal{A}(0)) \right) \cap C^1\left( [0,+\infty[, \mathcal{H} \right)$ for $U_0 \in D(\mathcal{A}(0))$.
The requested solution of \eqref{EQ2.17} is then given by
$$
U(t) = e^{\int_0^t \kappa(s)\,ds}\tilde{U}(t)
$$
because
\begin{align*}
U_t(t) &= \kappa(t)e^{\int_0^t \kappa(s)\,ds}\tilde{U}(t) + e^{\int_0^t \kappa(s)\,ds}\tilde{U}_t(t) \\
&= e^{\int_0^t \kappa(s)\,ds} \left( \kappa(t) + \tilde{\mathcal{A}}(t) \right) \tilde{U}(t) \\
&= \mathcal{A}(t)e^{\int_0^t \kappa(s)\,ds}\tilde{U}(t) \\
&=\mathcal{A}(t) U(t)
\end{align*}
which concludes the proof.
\end{proof}

\section{Exponential stability}\label{sec:Exponential_stability}

This section is dedicated to study of the asymptotic behavior. The main goal of this section is to study the stability of solutions to the system \eqref{EQ2.2}-\eqref{EQ2.3}. This is the content of Theorem \ref{Principal} where we show that the solution of problem \eqref{EQ2.2}-\eqref{EQ2.3} is exponentially stable. Our effort consists in building a suitable Lyapunov functional by the energy method. For this goal we present several technical lemmas.

\begin{lemma}
Let $(u,v,z)$ be a solution of \eqref{EQ2.2}-\eqref{EQ2.3}, then for any $\varepsilon_1 > 0$ and $c_1$ is the Poincar\'e's constant, we have the estimate
\begin{align}\label{y_2}
\frac{d}{dt}\mathcal{I}_1(t) \leq & - \left( a - \mu_1^2(0) c_1^2 \varepsilon_1  \right) \int_{\Omega} u_x^2\,dx - b \int_{L_1}^{L_2} v_x^2\,dx \\
& + \left( 1 + \frac{1}{2\varepsilon_1} \right) \int_{\Omega} u_t^2\,dx + \int_{L_1}^{L_2} v_t^2\,dx + \frac{\beta^2}{2 \varepsilon_1} \int_{\Omega} z^2(x,1,t)\,dx, \nonumber
\end{align}
where
\begin{equation}\label{y_1}
\mathcal{I}_1(t) = \int_{\Omega} uu_t\,dx + \int_{L_1}^{L_2} vv_t\,dx.
\end{equation}
\end{lemma}

\begin{proof}
Differentiating $\mathcal{I}_1(t)$ and using \eqref{EQ2.2}, we obtain
\begin{align*}
\frac{d}{dt}\mathcal{I}_1(t) & = \int_{\Omega} u_t^2\,dx - a \int_{\Omega} u_x^2\,dx - \mu_1(t) \int_{\Omega} u u_t\,dx - \mu_2(t) \int_{\Omega} u z(x,1,t)\,dx + \int_{L_1}^{L_2} v_t^2\,dx - b \int_{L_1}^{L_2} v_x^2\,dx \nonumber \\
&\leq \int_{\Omega} u_t^2\,dx - a \int_{\Omega} u_x^2\,dx + \left| \mu_1(t) \int_{\Omega} u u_t\,dx \right| + \left| \mu_2(t) \int_{\Omega} u z(x,1,t)\,dx \right| + \int_{L_1}^{L_2} v_t^2\,dx - b \int_{L_1}^{L_2} v_x^2\,dx.
\end{align*}
From hypothesis (H1) and (H2), we have
\begin{align}\label{s_1}
\frac{d}{dt}\mathcal{I}_1(t) \leq & \int_{\Omega} u_t^2\,dx - a \int_{\Omega} u_x^2\,dx + \mu_1(0) \left| \int_{\Omega} u u_t\,dx \right| \\
& + \beta \mu_1(0) \left| \int_{\Omega} u z(x,1,t)\,dx \right| + \int_{L_1}^{L_2} v_t^2\,dx - b \int_{L_1}^{L_2} v_x^2\,dx. \nonumber
\end{align}
By using the conditions \eqref{EQ2.2.1} and \eqref{EQ2.2.2}, we obtain
\begin{gather*}
u^2(x,t) = \left( \int_0^x u_x(s,t)\,ds \right)^2 \leq L_1 \int_{0}^{L_1} u_x^2(x,t)\,dx, \quad x \in [0,L_1], \\
u^2(x,t) \leq (L_3 - L_2) \int_{L_2}^{L_3} u_x^2(x,t)\,dx, \quad x \in [L_2,L_3],
\end{gather*}
which imply the following Poincar\'e's inequality
\begin{equation}\label{Poinc_1}
\int_{\Omega} u^2(x,t)\,dx \leq c_1^2 \int_{\Omega} u_x^2\,dx, \quad x \in \Omega,
\end{equation}
where $c_1 = \max \{ L_1, L_3 - L_2 \}$ is the Poincar\'e's constant. Using Young's inequality and \eqref{Poinc_1}, we have
\begin{equation}\label{s_2}
\mu_1(0) \left| \int_{\Omega} u u_t\,dx \right| \leq \frac{\varepsilon_1 \mu_1^2(0) c_1^2}{2} \int_{\Omega} u_x^2\,dx + \frac{1}{2\varepsilon_1} \int_{\Omega} u_t^2\,dx
\end{equation}
and
\begin{equation}\label{s_3}
\beta \mu_1(0)  \left| \int_{\Omega} u z(x,1,t)\,dx \right| \leq \frac{\varepsilon_1 \mu_1^2(0) c_1^2}{2} \int_{\Omega} u_x^2\,dx + \frac{\beta^2}{2\varepsilon_1} \int_{\Omega} z^2(x,1,t)\,dx.
\end{equation}
Substituting \eqref{s_2} and \eqref{s_3} in \eqref{s_1} we conclude the lemma.
\end{proof}

Now, inspired by \cite{Marzocchi_Naso_Rivera}, we introduce the functional
\begin{equation}
q(x) = \left\{ \begin{array}{lc}
x - \dfrac{L_1}{2}, & x \in [0, L_1], \\ \\
\dfrac{L_2 -L_3 -L_1}{2(L_2 - L_1)}(x - L_1) + \dfrac{L_1}{2}, & x \in [L_1, L_2], \\ \\
x - \dfrac{L_2 + L_3}{2}, & x \in [L_2, L_3].
\end{array} \right.
\end{equation}
It is easy to see that $q(x)$ is bounded, i.e., $\left| q(x) \right| \leq M$, where
$$
M = \max \left\{ \frac{L_1}{2}, \frac{L_3 - L_2}{2} \right\}.
$$
We have the following result.

\begin{lemma}
Let $(u,v,z)$ be a solution of \eqref{EQ2.2}-\eqref{EQ2.3}, then for any $\varepsilon_2 > 0$, the following estimates holds true
\begin{align}\label{y_4}
\frac{d}{dt}\mathcal{I}_2(t) \leq & \left( \frac{1}{2} + \frac{1}{2\varepsilon_2} \right) \int_{\Omega} u_t^2\,dx + \left( \frac{a}{2} + M^2 \mu_1^2(0) \varepsilon_2  \right) \int_{\Omega} u_x^2\,dx + \frac{\beta^2}{2 \varepsilon_2} \int_{\Omega} z^2(x,1,t)\,dx \\
& - \frac{1}{4} \left[ L_1 u_t^2(L_1,t) + (L_3-L_2)u_t^2(L_2,t) \right] - \frac{a}{4} \left[ L_1 u_x^2(L_1,t) + (L_3-L_2)u_x^2(L_2,t) \right], \nonumber
\end{align}
and
\begin{align}\label{y_5}
\frac{d}{dt}\mathcal{I}_3(t) = & \frac{L_2 - L_3 - L_1}{4(L_2 - L_1)} \left( \int_{L_1}^{L_2} v_t^2\,dx + b \int_{L_1}^{L_2} v_x^2\,dx \right) + \frac{1}{4} \left[ L_1 v_t^2(L_1,t) + (L_3-L_2)v_t^2(L_2,t) \right] \\
& + \frac{b}{4} \left[ L_1 v_x^2(L_1,t) + (L_3-L_2)v_x^2(L_2,t) \right], \nonumber
\end{align}
where
\begin{equation}\label{y_3}
\mathcal{I}_2(t) = - \int_{\Omega} q(x)u_x u_t\,dx \quad \mbox{and} \quad \mathcal{I}_3(t) = - \int_{L_1}^{L_2} q(x)v_x v_t\,dx.
\end{equation}
\end{lemma}

\begin{proof}
Differentiating $\mathcal{I}_2(t)$ and using \eqref{EQ2.2}, we obtain
\begin{align*}
\frac{d}{dt}\mathcal{I}_2(t) = & -\int_{\Omega} q(x)u_{xt} u_t\,dx - a \int_{\Omega} q(x)u_{xx} u_x\,dx \\
& + \mu_1(t) \int_{\Omega} q(x)u_{x} u_t\,dx + \mu_2(t) \int_{\Omega} q(x)u_{x} z(x,1,t)\,dx.
\end{align*}
Integrating by parts and considering the hypothesis (H1) and (H2), we have
\begin{align}\label{y_6}
\frac{d}{dt}\mathcal{I}_2(t) & \leq \frac{1}{2} \int_{\Omega} q'(x)u_t^2\,dx - \frac{1}{2} \left[ q(x) u_t^2 \right]_{\partial \Omega} + \frac{a}{2} \int_{\Omega} q'(x)u_x^2\,dx - \frac{a}{2} \left[ q(x) u_x^2 \right]_{\partial \Omega} \\
&\quad + \mu_1(0) \left| \int_{\Omega} q(x)u_{x} u_t\,dx \right| + \beta \mu_1(0) \left| \int_{\Omega} q(x)u_{x} z(x,1,t)\,dx \right| \nonumber \\
& \leq \frac{1}{2} \int_{\Omega} u_t^2\,dx - \frac{1}{2} \left[ q(x) u_t^2 \right]_{\partial \Omega} + \frac{a}{2} \int_{\Omega} u_x^2\,dx - \frac{a}{2} \left[ q(x) u_x^2 \right]_{\partial \Omega} \nonumber \\
&\quad + \mu_1(0)M \left| \int_{\Omega} u_{x} u_t\,dx \right| + \beta \mu_1(0) M \left| \int_{\Omega} u_{x} z(x,1,t)\,dx \right|. \nonumber
\end{align}
On the other hand, by using the boundary conditions \eqref{EQ2.2.2}, we have
\begin{gather*}
\frac{1}{2} \left[ q(x) u_t^2 \right]_{\partial \Omega} = \frac{1}{4} \left[ L_1 u_t^2(L_1,t) + (L_3-L_2) u_t^2(L_2,t) \right], \\
-\frac{a}{2} \left[ q(x) u_x^2 \right]_{\partial \Omega} \leq -\frac{a}{4} \left[ L_1 u_x^2(L_1,t) + (L_3-L_2) u_x^2(L_2,t) \right].
\end{gather*}
Inserting the above two equalities into \eqref{y_6} and by Young's inequality, we conclude that \eqref{y_6} gives \eqref{y_4}.

By the same argument, taking the derivative of $\mathcal{I}_3(t)$, we obtain
\begin{align*}
\frac{d}{dt}\mathcal{I}_3(t) & = \frac{1}{2} \int_{L_1}^{L_2} q'(x)v_t^2\,dx - \frac{1}{2} \left[ q(x) v_t^2 \right]_{L_1}^{L_2} + \frac{b}{2} \int_{L_1}^{L_2} q'(x)v_x^2\,dx - \frac{b}{2} \left[ q(x) v_x^2 \right]_{L_1}^{L_2} \\
& = \frac{L_2 - L_3 - L_1}{4(L_2 - L_1)} \left( \int_{L_1}^{L_2} v_t^2\,dx + b \int_{L_1}^{L_2} v_x^2\,dx \right) + \frac{1}{4} \left[ L_1 v_t^2(L_1,t) + (L_3-L_2)v_t^2(L_2,t) \right] \nonumber \\
&\quad + \frac{b}{4} \left[ L_1 v_x^2(L_1,t) + (L_3-L_2)v_x^2(L_2,t) \right]
\end{align*}
Hence, the proof is complete.
\end{proof}

As in \cite{Kirane_Said_Anwar}, taking into account the last lemma, we introduce the functional
\begin{equation}\label{y_7}
\mathcal{J}(t) = \bar{\xi} \tau(t) \int_{\Omega} \int_0^1 e^{-2\tau(t) \rho} z^2(x,\rho,t)\,d\rho\,dx.
\end{equation}

For this functional we have the following estimate.

\begin{lemma}[{\cite[Lemma~3.7]{Kirane_Said_Anwar}}]
Let $(u,v,z)$ be a solution of \eqref{EQ2.2}-\eqref{EQ2.3}. Then the functional $\mathcal{J}(t)$ satisfies
\begin{equation}\label{y_8}
\frac{d}{dt}\mathcal{J}(t) \leq -2 \mathcal{J}(t) + \bar{\xi} \int_{\Omega} u_t^2\,dx.
\end{equation}
\end{lemma}

Now we are in position to prove our result of stability.
\begin{theorem}\label{Principal}
Let $U(t) = (u(t),v(t),\varphi(t),\psi(t),z(t))$ be the solution of \eqref{EQ2.2}-\eqref{EQ2.3} with initial data $U_0 \in D\left( \mathcal{A}(0) \right)$ and $E(t)$ the energy of $U$. Assume that the hypothesis \eqref{hipot_1}, \eqref{hipot_2}, (H1), (H2) and
\begin{equation}\label{hipot_3}
\max \{1, \frac{a}{b}\} < \frac{L_1 + L_3 - L_2}{2(L_2 - L_1)}
\end{equation}
hold. Then there exist positive constants $c$ and $\alpha$ such that
\begin{equation}\label{theorem_principal}
E(t) \leq cE(0)e^{-\alpha t}, \quad \forall t \geq 0.
\end{equation}
\end{theorem}

\begin{proof}
Let us define the Lyapunov functional
\begin{equation}\label{y_9}
\mathcal{L}(t) = N E(t)(t) + \sum_{i=1}^{3}N_i \mathcal{I}_i(t) + \mathcal{J}(t),
\end{equation}
where $N$, $N_i$, $i=1,2,3$ are positive real numbers which will be chosen later. By the Lemma \ref{Lemma_2.2}, there exists a positive constant $K$ such that
\begin{equation}\label{y_10}
\frac{d}{dt}E(t) \leq -K \left[ \int_{\Omega} u_t^2\,dx + \int_{\Omega} z^2(x,1,t)\,dx \right].
\end{equation}
It follows from the transmission conditions \eqref{EQ2.2.1} that
\begin{equation}\label{y_11}
a^2 u_x^2(L_i,t) = b^2 v_x^2(L_i,t), \quad i=1,2.
\end{equation}
Using the estimates \eqref{y_2}, \eqref{y_4}, \eqref{y_5}, \eqref{y_8}, \eqref{y_10} and the equation \eqref{y_11}, we obtain
\begin{align}\label{y_12}
\frac{d}{dt}\mathcal{L}(t) \leq & - \left[ KN - \left( 1 + \frac{1}{2\varepsilon_1} \right)N_1 - \left( \frac{1}{2} + \frac{1}{2\varepsilon_2} \right)N_2 - \bar{\xi} \right] \int_{\Omega} u_t^2\,dx \\
& - \left( KN - \frac{\beta^2}{2\varepsilon_1}N_1 - \frac{\beta^2}{2\varepsilon_2}N_2 \right) \int_{\Omega} z^2(x,1,t)\,dx \nonumber \\
& - \left[ \left( a - \mu_1^2(0)c_1^2 \varepsilon_1 \right)N_1 - \left( \frac{a}{2} + M^2 \mu_1^2(0) \varepsilon_2 \right)N_2 \right] \int_{\Omega} u_x^2\,dx \nonumber \\
& + \left[ N_1 + \frac{L_2-L_3-L_1}{4(L_2-L_1)} N_3 \right] \int_{L_1}^{L_2} v_t^2\,dx \nonumber \\
& - \left[ N_1 - \frac{L_2-L_3-L_1}{4(L_2-L_1)} N_3 \right] b\int_{L_1}^{L_2} v_x^2\,dx \nonumber \\
& - \left(N_2 - N_3 \right) \left[ \frac{L_1}{4}u_t^2(L_1,t) + \frac{L_3-L_2}{4}u_t^2(L_2,t) \right] \nonumber \\
& - \left(N_2 - \frac{a}{b} N_3 \right) \frac{a}{4} \left[ \frac{L_1}{4}u_t^2(L_1,t) + \frac{L_3-L_2}{4}u_t^2(L_2,t) \right]  -2 \mathcal{J}(t). \nonumber
\end{align}

Now we observe that under assumption \eqref{hipot_3}, we can always find real constants $N_1, N_2$ and $N_3$ in such way that
\begin{equation*}
N_1 + \frac{L_2-L_3-L_1}{4(L_2-L_1)} N_3 < 0, \quad N_2 > \max \left\{ 1, \frac{a}{b} \right\} N_3, \quad N_1 > \frac{N_2}{2}.
\end{equation*}
After that, we pick positive constants $\varepsilon_1$ and $\varepsilon_2$ small enough that
\begin{equation*}
\mu_1^2(0)c_1^2 \varepsilon_1 N_1 + M^2 \mu_1^2(0) \varepsilon_2 N_2 < a \left( N_1 - \frac{N_2}{2} \right).
\end{equation*}
Finally, since $\xi(t)\tau(t)$ non-negative and limited, we choose $N$ large enough that \eqref{y_12} is taken into the  following estimate
\begin{align*}
\frac{d}{dt} \mathcal{L}(t) & \leq - \eta_1 \int_{\Omega} \left( u_t^2 + u_x^2 \right)\,dx - \eta_1 \int_{L_1}^{L_2} \left( v_t^2 + v_x^2 \right)\,dx - \eta_1 \int_{\Omega} z^2(x,\rho,t)\,dx - \eta_1 \int_{\Omega} z^2(x,1,t)\,dx \\
&\leq - \eta_1 \int_{\Omega} \left( u_t^2 + u_x^2 \right)\,dx - \eta_1 \int_{L_1}^{L_2} \left( v_t^2 + v_x^2 \right)\,dx - \eta_1 \int_{\Omega} z^2(x,\rho,t)\,dx,
\end{align*}
for a certain positive constant $\eta_1$.

This implies by \eqref{EQ2.8} that there exists $\eta_2 > 0$ such that
\begin{equation}\label{y_13}
\frac{d}{dt} \mathcal{L}(t) \leq - \eta_2 E(t), \quad \forall t \geq 0.
\end{equation}

On the hand, it is not hard to see for $N$ large enough that the $\mathcal{L}(t)\sim E(t)$, i.e. there exists two positive constants $\gamma_1$ and $\gamma_2$ such that
\begin{equation}\label{y_14}
\gamma_1 E(t) \leq \mathcal{L}(t) \leq \gamma_2 E(t), \quad \forall t \geq 0.
\end{equation}

Combining \eqref{y_13} and \eqref{y_14}, we obtain
\begin{equation*}
\frac{d}{dt} \mathcal{L}(t) \leq -\alpha \mathcal{L}(t), \quad \forall t \geq 0
\end{equation*}
which leads to
\begin{equation}\label{y_17}
\mathcal{L}(t) \leq \mathcal{L}(0)e^{-\alpha t}, \quad \forall t \geq 0.
\end{equation}
The desired result \eqref{theorem_principal} follows by using estimates \eqref{y_14} and \eqref{y_17}. Then, the proof of Theorem \ref{Principal} is complete.
\end{proof}

	\subsection*{Acknowledgment}
The authors thanks CAPES(Brazil).

\end{document}